\documentclass[11pt,a4paper]{article}
\date{}

\usepackage{amssymb}
\usepackage{latexsym}
\usepackage{amsmath,amssymb}
\usepackage{amsthm,enumerate,verbatim}
\usepackage{amsfonts}
\usepackage{graphicx}
\usepackage{xcolor}
\usepackage{subfigure}
\usepackage{multimedia}
\numberwithin{equation}{section}

\newtheorem{definition}{Definition}[section]
\newtheorem{theorem}{Theorem}[section]
\newtheorem{lemma}{Lemma}[section]
\newtheorem{algorithm}{Algorithm}[section]

\newtheorem{remark}{Remark} [section]

\newcommand{\R}{{\mathbb{R}}}
\newcommand{\IN}{{\mathbb{N}}}
\newcommand{\dist}{{\rm dist}}

\title{Convergence Rates of Subgradient Methods for Quasi-convex Optimization Problems}
\author{Yaohua Hu\thanks{Shenzhen Key Laboratory of Advanced Machine Learning and Applications, College of Mathematics and Statistics, Shenzhen University, Shenzhen 518060, P. R. China (mayhhu@szu.edu.cn). This author's work was supported in part by the National Natural Science Foundation of China (11601343, 11871347), Natural Science Foundation of Guangdong (2016A030310038), Natural Science Foundation of Shenzhen (JCYJ20170817100950436, JCYJ20170818091621856) and Interdisciplinary Innovation Team of Shenzhen University.},
\quad Jiawen Li\thanks{College of Mathematics and Statistics, Shenzhen University, Shenzhen 518060, P. R. China (654439697@qq.com).},
\quad Carisa Kwok Wai Yu\thanks{Department of Mathematics and Statistics, The Hang Seng University of Hong Kong, Hong Kong (carisayu@hsu.edu.hk). This author's work was supported in part by grants from the Research Grants Council of the Hong Kong Special Administrative Region, China (UGC/FDS14/P02/15 and UGC/FDS14/P02/17).}
}

\begin{document}

\maketitle

\noindent {\bf Abstract}\quad
Quasi-convex optimization acts a pivotal part in many fields including economics and finance; the subgradient method is an effective iterative algorithm for solving large-scale quasi-convex optimization problems.
In this paper, we investigate the iteration complexity and convergence rates of various subgradient methods for solving quasi-convex optimization in a unified framework. In particular, we consider a sequence satisfying a general (inexact) basic inequality, and investigate the global convergence theorem and the iteration complexity when using the constant, diminishing or dynamic stepsize rules. More important, we establish the linear (or sublinear) convergence rates of the sequence under an additional assumption of weak sharp minima of H\"{o}lderian order and upper bounded noise.
These convergence theorems are applied to establish the iteration complexity and convergence rates of several subgradient methods, including the standard/inexact/conditional subgradient methods, for solving quasi-convex optimization problems under the assumptions of the H\"{o}lder condition and/or the weak sharp minima of H\"{o}lderian order.




\section{Introduction}
Mathematical optimization is a fundamental tool for solving decision-making problems in many disciplines. Convex optimization plays a key role in mathematical optimization, but may not be applicable to some problems encountered in economics, finance and management science. In contrast to convex optimization, quasi-convex optimization usually provides a much more accurate representation of realities and still possesses some desirable properties of convex optimization. In recent decades, more and more attention has been drawn to quasi-convex optimization; see \cite{Avriel10,CrouzeixMM98,HaKS05,Mastrogiacomo2015,Stancu97} and references therein.

It is essential to develop efficient numerical algorithms for quasi-convex optimization problems. Subgradient methods form a class of practical and effective iterative algorithms for solving constrained (convex or quasi-convex) optimization problems. The subgradient method was originally introduced by Polyak \cite{Polyak67} and Ermoliev \cite{Ermoliev66} in the 1970s to solve nondifferentiable convex optimization problems. Over the past 40 years, various features of subgradient methods have been established for convex optimization problems and many extensions/generalizations have been devised for structured convex optimization problems; see \cite{Auslender09,Bertsekas99,HuSIOPT16,Kiwiel04,LarssonPS96,LuZL2017,Nedic01,Nesterov09,Nesterov18,Shor85,HuIP17} and references therein. Moreover, subgradient methods have also been extended and developed to solve constrained quasi-convex optimization problems; see \cite{HuEJOR15, HuJNCA16b,Kiwiel01,Konnov03,Quiroz2015,HuOpt18} and references therein. The global convergence properties of subgradient methods, in terms of function values and distances of iterates from the optimal solution set, have been well studied for either convex or quasi-convex optimization problems by using several typical stepsize rules.

In addition to the global convergence property, the establishment of convergence rates is important in analyzing the numerical performance of relevant algorithms. The convergence rates of subgradient methods for solving convex optimization problems have been investigated under the assumption of weak sharp minima \cite{BurkeF93}. For example, employing a geometrically decaying stepsize, Shor \cite[Theorem 2.7]{Shor85} and Goffin \cite[Theorem 4.4]{Goffin1977} provided a linear convergence analysis (but not necessarily converge to an optimal solution) of the subgradient method under a notion of condition number that is a stronger condition than the weak sharp minima. This work was extended by Polyak \cite[Theorem 4]{Poljak1978} to the case when the subgradients are corrupted by deterministic noise. When using a dynamic stepsize rule, Brannlund \cite[Theorem 2.5]{Brannlund1995} established the linear convergence rate of the subgradient method under the assumption of weak sharp minima, which was generalized by Robinson \cite[Theorem 3]{Robinson1999} to the inexact subgradient method. Recently, for vast applications of distributed optimization, an incremental subgradient method was proposed to solve the convex sum-minimization problem and was shown to converge linearly under the assumption of weak sharp minima (see \cite[Proposition 2.11]{NedicBert2001}); the linear convergence analysis has also been extended to various variant incremental subgradient methods; see \cite{IACA2017,Mokhtari2018} and references therein. It is worth noting that all the convergence rates of subgradient methods are characterized in terms of distances of iterates from the optimal solution set. Moreover, Freund and Lu \cite{Freund2018} and Xu et al. \cite{XuLY2018} investigated the iteration complexity (for achieving an approximate optimal value) of subgradient methods and stochastic subgradient methods under the assumption of weak sharp minima or weak sharp minima of H\"{o}lderian order, respectively. 
However, to the best of our knowledge, there is still limited study devoted to establishing the convergence rates of subgradient methods in the scenario of quasi-convex optimization.

Motivated by the wide applications of quasi-convex optimization, we contribute to the further development of subgradient methods for solving quasi-convex optimization problems, particularly concentrating on the iteration complexity and convergence rate analysis. In the present paper, we consider the following constrained quasi-convex optimization problem
\begin{equation}\label{eq-QCO}
  \begin{array}{ll}
     \text{min}&f(x)\\
     \text{s.t.}  & x\in X,
    \end{array}
 \end{equation}
where $f:\R^n\rightarrow \R$ is a quasi-convex and continuous function, and $X\subseteq \R^n$ is a nonempty, closed and convex set. Kiwiel \cite{Kiwiel01} proposed a (standard) subgradient method to solve problem \eqref{eq-QCO}, which is formally stated as follows, where $\partial^* f$, $\mathbf{S}$ and $\mathbb{P}_X(\cdot)$ denote the quasi-subdifferential of $f$, the unit sphere centered at the origin and the Euclidean projection onto $X$, respectively.
\begin{algorithm}\label{alg-SGM}
{\rm
Select an initial point $x_1\in \R^n$ and a sequence of stepsizes $\{v_k\}\subseteq (0,+\infty)$. For each $k\in \IN$, having $x_k$, we select
$g_k\in \partial^* f(x_k)\cap \mathbf{S}$ and update $x_{k+1}$ by
\[x_{k+1}:=\mathbb{P}_X (x_k-v_k g_k).\]
}
\end{algorithm}
Note that a basic inequality of a subgradient iteration is a key tool for the convergence analysis of subgradient methods for either convex or quasi-convex optimization. Yu et al. \cite{HuOpt18} developed a unified framework of convergence analysis for various subgradient methods, in which the global convergence theorem was established for a certain class of sequences satisfying a general basic inequality and using several typical stepsize rules. In real applications, the computation error stems from practical considerations and is inevitable in the computing process. To meet the requirements of practical applications, Hu et al. \cite{HuEJOR15} introduced an inexact subgradient method to solve problem \eqref{eq-QCO} and investigated the influence of the deterministic noise on the inexact subgradient method. In the present paper, motivated by the practical applications, we consider a more general unified framework for subgradient methods, in which a general (inexact) basic inequality is assumed to be satisfied; see conditions (H1) and (H2) in Section 3. The more general unified framework covers various subgradient methods (see \cite{HuOpt18}) and inexact subgradient methods (see Section 4.2) for either convex or quasi-convex optimization problems.

The main contribution of the present paper is to investigate the iteration complexity and convergence rates of several subgradient methods for solving quasi-convex optimization problem \eqref{eq-QCO} via a unified framework. The stepsize rules adopted in this paper are the constant, diminishing and dynamic stepsizes. Here, we first consider the sequence satisfying conditions (H1) and (H2), establish the global convergence theorem in terms of function values and distances of iterates from the optimal solution set (see Theorem \ref{thm-abs}), and then derive the iteration complexity to obtain an approximate optimal solution (see Theorem \ref{thm-ce-abs} and Remark \ref{rem-ci}). More important, we explore the linear (or sublinear) convergence rates of the sequence under an additional assumption of weak sharp minima of H\"{o}lderian order and the upper bounded noise; see Theorems \ref{thm-abs-con}-\ref{thm-abs-dimi} for details.
Meanwhile, we will apply the established theorems on iteration complexity and convergence rates to investigate that of several subgradient methods for solving quasi-convex optimization problem \eqref{eq-QCO} when using the typical stepsize rules and under the assumptions of the H\"{o}lder condition and/or the weak sharp minima of H\"{o}lderian order. As far as we know, the study of iteration complexity and convergence rates theorems of subgradient methods are new in the literature of quasi-convex optimization.

In a very recent study by Johnstone and Moulin \cite{Johnstone2019}, by virtue of the weak sharp minima of H\"{o}lderian order, they derived the convergence rates of the standard subgradient method for solving convex optimization problems when using the constant-type stepsize rules (including the constant and ``descending stairs" stepsize rules). Our work extends \cite{Johnstone2019} to more general settings with quasi-convex optimization, inexact setting, diminishing and dynamic stepsize rules.

The present paper is organized as follows. In Section 2, we present the notations and preliminary results which will be used in this paper. In Section 3, we provide a unified framework of convergence analysis for (inexact) subgradient methods, in which the global convergence theorem, iteration complexity and convergence rates are established for a certain class of sequences satisfying a general (inexact) basic inequality under the assumption of weak sharp minima of H\"{o}lderian order. In Section 4, the convergence analysis framework is applied to establish the iteration complexity and convergence rates for the standard subgradient method, the inexact subgradient method and the conditional subgradient method for solving quasi-convex optimization problems.

\section{Notations and preliminary results}
The notations used in the present paper are standard; see, e.g., \cite{Bertsekas99}. We consider the $n$-dimensional Euclidean space $\R^n$ with inner product $\langle \cdot,\cdot\rangle$ and norm $\|\cdot\|$. For $x\in \R^n$ and $r>0$, we use $\mathbb{B}(x,r)$ to denote the closed ball centered at $x$ with radius $r$, and use $\mathbf{S}$ to denote the unit sphere centered at the origin.
For $x\in \R^n$ and $Z\subseteq \R^n$, the Euclidean distance of $x$ from $Z$ and the Euclidean projection of $x$ onto $Z$ are respectively defined by
\begin{equation*}
\dist(x, Z):=\min_{z\in Z}\|x-z\|\quad {\rm and} \quad \mathbb{P}_Z(x):={\rm arg}\min_{z\in Z} \|x-z\|.
\end{equation*}
A function $f:\R^n\rightarrow \R$ is said to be quasi-convex if for each $x,y\in \R^n$ and each $\alpha\in [0,1]$, the following inequality holds
\begin{equation*}
f((1-\alpha)x+\alpha y)\le \max\{f(x),f(y)\}.
\end{equation*}
For $\alpha\in\R$, the sublevel sets of $f$ are denoted by
\[
{\rm lev}_{<\alpha}f:=\{x\in \R^n :f(x)<\alpha\}\quad {\rm and} \quad {\rm lev}_{\le \alpha}f:=\{x\in \R^n :f(x)\le\alpha\}.
\]
It is well-known that $f$ is quasi-convex if and only if ${\rm lev}_{<\alpha}f$ $\left({\rm and/or}~{\rm lev}_{\le \alpha}f\right)$ is convex for each $\alpha\in\R$. A function $f:\R^n\rightarrow \R$ is said to be coercive if $\lim_{\|x\|\to \infty} f(x)=\infty$, and so its sublevel set ${\rm lev}_{\le \alpha}f$ is bounded for each $\alpha\in\R$.

The convex subdifferential $\partial f(x):=\{g\in \R^n: f(y)\ge f(x)+\langle g,y-x\rangle, \forall y\in \R^n\}$ might be empty for the quasi-convex functions. Hence, the introduction of (nonempty) subdifferential of quasi-convex functions plays an important role in quasi-convex optimization. Several different types of subdifferentials of quasi-convex functions have been introduced in the literature, see \cite{Aussel00,GrP73,HuEJOR15,Kiwiel01} and references therein. In particular, Kiwiel \cite{Kiwiel01} and Hu et al \cite{HuEJOR15} introduced a quasi-subdifferential and applied this quasi-subgradient in their proposed subgradient methods; see, e.g.,  \cite{HuEJOR15,HuJNCA16b,Kiwiel01}.

\begin{definition}
Let $f:\R^n \to \R$ be a quasi-convex function and let $\epsilon>0$.
The quasi-subdifferential and $\epsilon$-quasi-subdifferential of $f$ at $x\in \R^n$ are respectively defined by
\begin{equation*}\label{eq-def-QS}
\partial^* f(x):=\left\{g\in \R^n:\langle g,y-x\rangle \le 0, \forall y\in {\rm lev}_{<f(x)}f\right\},
\end{equation*}
and
\begin{equation*}\label{eq-def-QS-epsilon}
\partial^*_{\epsilon} f(x):=\left\{g\in \R^n:\langle g,y-x\rangle \le 0, \forall y\in {\rm lev}_{<f(x)-\epsilon}f\right\}.
\end{equation*}
Any vector $g\in \partial^* f(x)$ or $g\in \partial^*_{\epsilon} f(x)$ is called a quasi-subgradient or an $\epsilon$-quasi-subgradient of $f$ at $x$, respectively.
\end{definition}
It is clear from definition that the quasi-subdifferential is a normal cone to the strict sublevel set of the quasi-convex function, and it was shown in \cite[Lemma 2.1]{HuEJOR15} that $\partial^* f(x)\setminus \{0\}\neq \emptyset$ whenever $f$ is quasi-convex. Hence, the quasi-subdifferential of a quasi-convex function contains at least a unit vector. This is a special property that the convex subdifferential does not share. In particular, it was claimed in \cite{HuEJOR15} that the quasi-subdifferential coincides with the convex cone hull of the convex subdifferential whenever $f$ is convex.

The H\"{o}lder condition (restricted to the set of minima) was used in \cite{Konnov94} to describe some properties of the quasi-subgradient, and it plays a critical role in the study of convergence analysis in quasi-convex optimization; see, e.g., \cite{HuEJOR15,HuOpt18}. The notion of H\"{o}lder condition has been widely studied and applied in harmonic analysis, fractional analysis and management science; see, e.g., \cite{Aussel00,Stancu97}. It is worth noting that the classical Lipschitz condition (i.e., H\"{o}lder condition of order 1) is equivalent to the bounded subgradient assumption, which is always assumed in the literature of subgradient methods (see, e.g., \cite{Bertsekas99,Kiwiel04,Shor85}), whenever $f$ is convex. $f^*$ and $X^*$ denote the minimum value and the set of minima of problem \eqref{eq-QCO}, respectively.
\begin{definition}\label{def-HC}
Let $p\in (0,1]$ and $L>0$. The function $f:\R^n\rightarrow \R$ is said to satisfy the H\"{o}lder condition (restricted to the set of minima $X^*$) of order $p$ with modulus $L$ on $\R^n$ if
\begin{equation*}\label{eq-HC}
f(x)-f^*\le L\dist^p(x, X^*)\quad \mbox{for each } x\in \R^n.
\end{equation*}
\end{definition}

%

We end this section by the following lemmas, which are useful to establish the unified framework of convergence analysis. In particular, Lemmas \ref{lem-averaging} is taken from \cite[Lemma 2.1]{Kiwiel04}.

\begin{lemma}\label{lem-averaging}
Let $\{a_k\}$ be a scalar sequence and let $\{w_k\}$ be a sequence of nonnegative scalars.
Suppose that $\lim_{k\to \infty} \sum_{i=1}^k w_i= \infty$. Then, it holds that
\[
\liminf_{k\to \infty} a_k\le \liminf_{k\to \infty} \frac{\sum_{i=1}^k w_i a_i}{\sum_{i=1}^k w_i}
\le \limsup_{k\to \infty} \frac{\sum_{i=1}^k w_i a_i}{\sum_{i=1}^k w_i}\le \limsup_{k\to \infty}a_k.
 \]
\end{lemma}


\begin{lemma}\label{lem-new}
Let $r>0$, $a>0$, $b\ge 0$, and let $\{u_k\}$ be a sequence of nonnegative scalars such that
\begin{equation}\label{eq-lem-new1}
u_{k+1}\le u_k-a u_k^{1+r} +b \quad \mbox{for each } k\in \IN.
\end{equation}
\begin{enumerate}[{\rm (i)}]
  \item If $b=0$, then
    \[
    u_{k+1}\le u_1\left(1+rau_1^rk\right)^{-\frac1r}\quad \mbox{for each } k\in \IN.
    \]
  \item If $0<b<a^{-\frac1r}(1+r)^{-\frac{1+r}r}$, then there exists $\tau\in (0,1)$ such that
    \[
    u_{k+1}\le u_1\tau^k  + \left(\frac{b}a\right)^{\frac1{1+r}} \quad \mbox{for each } k\in \IN.
    \]
\end{enumerate}
\end{lemma}
\begin{proof}
Assertion (i) of this lemma is taken from \cite[pp. 46, Lemma 6]{Polyak87}, then it remains to prove assertion (ii). For this purpose, let $u:=\left(\frac{b}a\right)^{\frac1{1+r}}$. Then, \eqref{eq-lem-new1} is reduced to
\begin{equation}\label{eq-lem-new2}
u_{k+1}-u\le u_k-u - a \left(u_k^{1+r} -u^{1+r}\right) \quad \mbox{for each } k\in \IN.
\end{equation}
As $r>0$, by the convexity of $h(t):=t^{1+r}$, one has that $u_k^{1+r} -u^{1+r}\ge (1+r)u^r(u_k-u)$. Then \eqref{eq-lem-new2} implies that
\[
u_{k+1}-u\le \left(1-a(1+r)u^r\right)(u_k-u) \quad \mbox{for each } k\in \IN.
\]
Let $\tau:=1-a(1+r)u^r$. It follows from the assumption that $\tau\in (0,1)$, and thus,
\[
u_{k+1}-u\le \tau^k (u_1-u) \quad \mbox{for each } k\in \IN.
\]
The conclusion follows and the proof is complete.
\end{proof}

\begin{lemma}\label{lem-dimi}
Let $a>0$, $b>0$, $s \in (0,1)$ and $t\ge s$, and let $\{u_k\}$ be a sequence of nonnegative scalars such that
\begin{equation}\label{eq-lem-new3}
u_{k+1}\le \left(1-ak^{-s}\right) u_k+bk^{-t}\quad \mbox{for each } k\in \IN.
\end{equation}
\begin{enumerate}[{\rm (i)}]
    \item If $t>s$, then
    \[
    u_{k+1}\le \frac{b}{a}k^{s-t} +o(k^{s-t}).
    \]
  \item If $t=s$, then there exists $\tau\in (0,1)$ such that
    \[
    u_{k+1}\le u_1 e^{\frac{as}{1-s}} \tau^k +\frac{b}a  \quad \mbox{for each } k\in \IN.
  \]
\end{enumerate}
\end{lemma}
\begin{proof}
Assertion (i) of this lemma is taken from \cite[pp. 46, Lemma 5]{Polyak87}, then it remains to prove assertion (ii). To this end, we derive by \eqref{eq-lem-new3} (when $t=s$) that, for each $k\in \IN$,
\begin{equation}\label{eq-lem-new4}
u_{k+1}-\frac{b}a\le \left(1-ak^{-s}\right) \left(u_k-\frac{b}a\right)\le \left(u_1-\frac{b}a\right)  \prod_{i=1}^{k} (1-ai^{-s}).
\end{equation}
Note that
\[
\prod_{i=1}^{k} (1-ai^{-s})=e^{\sum_{i=1}^{k}\ln(1-ai^{-s})} < e^{\int_{1}^{k+1}\ln(1-at^{-s}){\rm d}t} < e^{\int_{1}^{k+1}-at^{-s}{\rm d}t}\le e^{\frac{as}{1-s}} \tau^k,
\]
where $\tau:=e^{-a}\in (0,1)$. This, together with \eqref{eq-lem-new4}, yields the conclusion.
\end{proof}

\section{A unified framework for subgradient methods}
In the present paper, we discuss subgradient methods for solving the quasi-convex optimization problem \eqref{eq-QCO}, in which the set of minima and the minimum value are denoted by $X^*$ and $f^*$, respectively. The class of subgradient methods is one of the most popular numerical algorithms for solving constrained optimization problems. In view of the procedure of subgradient methods, the basic inequality of a subgradient iteration is an important property and plays as a key tool for convergence analysis of subgradient methods for either convex or quasi-convex optimization problems.

This section aims to investigate the iteration complexity and convergence rates of subgradient methods via a unified framework, in which a general (inexact) basic inequality is assumed to be satisfied. In particular, we fix $\epsilon\ge 0$ and $p\in (0,1]$, and consider a sequence $\{x_k\}\subseteq X$ that satisfies the following two conditions:
\begin{enumerate}
  \item[(H1)] For each $x^*\in X^*$ and each $k\in \{i\in \IN:f(x_i)>f^*+\epsilon\}$,
  \begin{equation}\label{eq-abs}
  \|x_{k+1}-x^*\|^2 - \|x_k-x^*\|^2 \le -\alpha_kv_k (f(x_k)-f^*-\epsilon)^{\frac1p}+ \beta_kv_k^2.
  \end{equation}
  \item[(H2)] $\{\alpha_k\}$ and $\{\beta_k\}$ are two sequences of positive scalars such that
  \begin{equation}\label{eq-abs-0}
  \lim_{k\to \infty}\alpha_k=\alpha>0\quad \mbox{and}\quad \lim_{k\to \infty}\beta_k=\beta>0.
  \end{equation}
\end{enumerate}

Condition ${\rm (H1)}$ measures the difference between two distances of iterates from a possible solution by calculating the difference between the function value and the minimum value with a noise, and condition ${\rm (H2)}$ characterizes some assumptions on the parameters. In the special case when $\epsilon=0$, conditions ${\rm (H1)}$ and ${\rm (H2)}$ are reduced to the unified framework studied in \cite{HuOpt18}, where the global convergence theorem was established, but no convergence rate analysis.
The nature of subgradient methods forces the generated sequence to comply with conditions ${\rm (H1)}$ and ${\rm (H2)}$ under some mild assumptions, and thus, this study provides a unified framework for various subgradient methods for either convex or quasi-convex optimization problems (one can also refer to \cite{HuOpt18} for details).
\begin{itemize}
  \item[-] For convex optimization problems and under a bounded subgradient assumption, condition ${\rm (H1)}$ with $p=1$ and ${\rm (H2)}$ are satisfied for the subgradient-type methods, including the standard subgradient method \cite{Shor85}, the approximate subgradient method \cite{Kiwiel04}, the primal-dual subgradient method \cite{Nedic09Saddle}, the incremental subgradient method \cite{Nedic01}, the conditional subgradient method \cite{LarssonPS96} and a unified framework of subgradient methods \cite{NetoPierro09};
  \item[-] For quasi-convex optimization problems and under the assumption of H\"{o}lder condition of order $p$, conditions ${\rm (H1)}$ and ${\rm (H2)}$ are satisfied for several types of subgradient methods, such as the standard subgradient method \cite{Kiwiel01}, the inexact subgradient method \cite{HuEJOR15}, the primal-dual subgradient method \cite{HuPAFA17} and the conditional subgradient method \cite{HuJNCA16b}.
\end{itemize}
\subsection{Convergence theorem}
This subsection aims to establish the convergence theorem of the sequence satisfying conditions ${\rm (H1)}$ and ${\rm (H2)}$ for some suitable stepsize rules $\{v_k\}$. The stepsize rule has a critical impact on the convergence behavior and numerical performance of subgradient methods. In the present paper, we consider three typical stepsize rules: (i) the constant stepsize rule is the most popular in applications but only guarantees the convergence to the optimal value/solution set within some tolerance; (ii) the diminishing stepsize rule guarantees the convergence to the exact optimal value/solution set but suffers a slow convergence rate; (iii) the dynamic stepsize rule enjoys the best convergence property but requires prior information of the approximate optimal value $f^*+\epsilon$; see \cite{Nedic01,HuIP17,HuOpt18} and references therein. Theorem \ref{thm-abs} extends \cite[Theorem 3.1]{HuOpt18} (considering the special case when $\epsilon=0$) to the inexact setting, while the skeleton of the proof is similar to that of \cite[Theorem 3.1]{HuOpt18}. To make this paper more self-contained, we provide a proof of the convergence theorem as follows. We write $X^*_\epsilon:=X\cap {\rm lev}_{\le f^*+\epsilon}f$ for the sake of simplicity, and particularly, $X^*_0=X^*$.
\begin{theorem}\label{thm-abs}
Let $\{x_k\}\subseteq X$ satisfy ${\rm (H1)}$ and ${\rm (H2)}$. Then, the following assertions are true.
\begin{enumerate}[{\rm (i)}]
  \item If $v_k\equiv v>0$, then $\liminf_{k\to \infty} f(x_k)\le f^*+\left(\frac{\beta v}{\alpha}\right)^p+\epsilon$.
  \item If $\{v_k\}$ is given by
    \begin{equation}\label{eq-dim+stepsize}
    v_k:= c k^{-s},\quad \mbox{where $c>0$, $s\in (0,1)$},
    \end{equation}
    then $\liminf_{k\to \infty} f(x_k)\le f^*+\epsilon$.
  \item If $\{v_k\}$ is given by
  \begin{equation}\label{eq-dyn+stepsize}
    v_k:= \frac{\alpha_k\lambda_k}{2\beta_k}[f(x_k)-f^*-\epsilon]_+^{\frac1p},\quad \mbox{where $0<\underline{\lambda}\le \lambda_k \le \overline{\lambda}<2$},
  \end{equation}
  then either $x_k\in X^*_\epsilon$ for some $k\in \IN$ or $\lim_{k\to \infty} f(x_k)\le f^*+\epsilon$.
\end{enumerate}
\end{theorem}
\begin{proof}
Without loss of generality, we assume that $f(x_k)\le f^*+\epsilon$ only occurs for finitely many times; otherwise, assertions (i) and (ii) of this theorem hold automatically. That is, there exists $N\in \IN$ such that $f(x_k)>f^*+\epsilon$ for each $k\ge N$; consequently, letting $x^*\in X^*$, (H1) indicates that
\[
\|x_{k+1}-x^*\|^2 - \|x_k-x^*\|^2 \le -\alpha_kv_k (f(x_k)-f^*-\epsilon)^{\frac1p}+ \beta_kv_k^2
\]
for each $k\ge N$. Summing the above inequality over $k=N,\dots,n$, we have
\begin{equation*}
\|x_{n+1}-x^*\|^2-\|x_N-x^*\|^2\le -\sum_{k=N}^n \alpha_kv_k(f(x_k)-f^*-\epsilon)^{\frac{1}{p}}+\sum_{k=N}^n \beta_kv_k^2,
\end{equation*}
that is,
\begin{equation}\label{eq-abs-a1}
\frac{\sum_{k=N}^n \alpha_kv_k(f(x_k)-f^*-\epsilon)^{\frac{1}{p}}}{\sum_{k=N}^n\alpha_kv_k}\le \frac{\|x_N-x^*\|^2}{\sum_{k=N}^n\alpha_kv_k}+\frac{\sum_{k=N}^n\beta_kv_k^2}{\sum_{k=N}^n\alpha_kv_k}.
\end{equation}

(i) By the constant stepsize rule and \eqref{eq-abs-0}, one has $\lim_{n\to \infty} \sum_{k=N}^n \alpha_kv_k= \infty$. Then, by \eqref{eq-abs-a1}, Lemma \ref{lem-averaging} is applicable (with $(f(x_k)-f^*-\epsilon)^{\frac{1}{p}}$ and $\alpha_kv_k$ in place of $a_k$ and $w_k$) to concluding that
\begin{eqnarray}
\liminf_{k\to \infty} \,(f(x_k)-f^*-\epsilon)^{\frac{1}{p}}
&\le& \liminf_{n\to \infty}\, \frac{\sum_{k=N}^n \alpha_kv_k(f(x_k)-f^*-\epsilon)^{\frac{1}{p}}}{\sum_{k=N}^n\alpha_kv_k}\nonumber \\
&\le& \liminf_{n\to \infty}\, \left(\frac{\|x_N-x^*\|^2}{\sum_{k=N}^n\alpha_kv_k}+\frac{\sum_{k=N}^n\beta_kv_k^2}{\sum_{k=N}^n\alpha_kv_k}\right). \label{eq-abs-con-a}
\end{eqnarray}
Note by \eqref{eq-abs-0} and Lemma \ref{lem-averaging} that
\begin{equation*}
\lim_{n\to \infty}\, \frac{\|x_N-x^*\|^2}{\sum_{k=N}^n\alpha_kv_k} = 0 \quad \mbox{and} \quad \lim_{n\to \infty}\, \frac{\sum_{k=N}^n\beta_kv_k^2}{\sum_{k=N}^n\alpha_kv_k}=\frac{\beta v}{\alpha}.
\end{equation*}
This, together with \eqref{eq-abs-con-a}, shows that $\liminf_{k\to \infty} \,(f(x_k)-f^*-\epsilon)^{\frac{1}{p}}\le \frac{\beta v}{\alpha}$,
and hence assertion (i) of this theorem is proved.

(ii) By \eqref{eq-abs-0} and \eqref{eq-dim+stepsize}, one has $\lim_{n\to \infty} \sum_{k=N}^n \alpha_kv_k= \infty$; consequently, \eqref{eq-abs-con-a} holds. Note by \eqref{eq-abs-0}, \eqref{eq-dim+stepsize} and Lemma \ref{lem-averaging} that
\begin{equation*}
\lim_{n\to \infty}\, \frac{\|x_N-x^*\|^2}{\sum_{k=N}^n\alpha_kv_k} = 0 \quad \mbox{and} \quad \lim_{n\to \infty}\, \frac{\sum_{k=N}^n\beta_kv_k^2}{\sum_{k=N}^n\alpha_kv_k}=0.
\end{equation*}
This, together with \eqref{eq-abs-con-a}, yields that $\liminf_{k\to \infty} f(x_k)\le f^*+\epsilon$, as desired.

(iii) Without loss of generality, we assume that $f(x_k)> f^*+\epsilon$ for each $k\in \IN$; otherwise, assertion (iii) of this theorem follows. By \eqref{eq-abs-0}, there exists $N\in \IN$ such that
\begin{equation*}\label{eq-dimi-1an}
\beta_k<2\beta \quad \mbox{and} \quad \alpha_k>\frac{\alpha}2 \quad \mbox{for each } k\ge N.
\end{equation*}
Then, for each $x^*\in X^*$, it follows from \eqref{eq-abs} and \eqref{eq-dyn+stepsize} that, for each $k\ge N$,
\begin{eqnarray}
\|x_{k+1}-x^*\|^2-\|x_k-x^*\|^2
&\le& -\frac{\alpha_k^2}{4\beta_k}\lambda_k(2-\lambda_k)(f(x_k)-f^*-\epsilon)^{\frac{2}{p}}\nonumber \\
&\le& -\frac{\alpha^2}{32\beta}\underline{\lambda}(2-\overline{\lambda})(f(x_k)-f^*-\epsilon)^{\frac{2}{p}}. \label{eq-abs-dyn}
\end{eqnarray}
This shows that $\lim_{k\to \infty} f(x_k)\le f^*+\epsilon$; otherwise, it follows from \eqref{eq-abs-dyn} that there exists $\sigma>0$ such that $\|x_{k+1}-x^*\|^2\le \|x_{k}-x^*\|^2-\sigma$ for infinitely many $k\ge N$, which is impossible (as $\{\|x_k-x^*\|\}$ is nonnegative).
\end{proof}

\begin{remark}
It is worth noting that the conclusion of Theorem \ref{thm-abs}(ii) is also true for the general diminishing stepsize rule, i.e., satisfying
\begin{equation*}
v_k>0,\quad \lim_{k\to \infty} v_k=0,\quad \sum_{k=0}^{\infty} v_k=\infty.
\end{equation*}
\end{remark}

\subsection{Iteration complexity}
This subsection is devoted to the complexity issue of the sequence satisfying conditions ${\rm (H1)}$ and ${\rm (H2)}$ when using the typical stepsize rules. Given $\delta>0$, the iteration complexity of a particular algorithm is to estimate the number of iterations $K$ required by the algorithm to obtain an approximate solution, at which the function value is within $\delta$ of the optimal, i.e.,
\[
\min_{1\le k\le K} f(x_k)\le f^*+\delta.
\]
We write
\begin{equation*}\label{eq-alpha-inf}
\alpha_{\inf}:=\inf_{k\in \IN} \,\alpha_k\quad \mbox{and} \quad \beta_{\sup}:=\sup_{k\in \IN} \,\beta_k
\end{equation*}
for simplicity.
It is clear that $\alpha_{\inf}\in (0,+\infty)$ and $\beta_{\sup}\in (0,+\infty)$ under the assumption ${\rm (H2)}$.

\begin{theorem}\label{thm-ce-abs}
Let $\delta>0$, and let $\{x_k\}\subseteq X$ satisfy ${\rm (H1)}$ and ${\rm (H2)}$.
\begin{enumerate}[{\rm (i)}]
  \item Let $K_1:=\frac{\dist^2(x_1,X^*)}{\alpha_{\inf} v \delta}$ and $v_k\equiv v>0$. Then
  \[\min\limits_{1\le k\le K_1} f(x_k)\le f^*+\left(\frac{\beta_{\sup} }{\alpha_{\inf}}v+\delta\right)^p+\epsilon.\]
  \item Let $K_2:=\left(\frac{(1-s)\dist^2(x_1,X^*)}{\alpha_{\inf} c \delta}\right)^{\frac1{1-s}}$ and $\{v_k\}$ be given by \eqref{eq-dim+stepsize}. Then
  \[\min\limits_{1\le k\le K_2} f(x_k)\le f^*+\left(\frac{\beta_{\sup} }{\alpha_{\inf}}ck^{-s}+\delta\right)^p+\epsilon.\]
  \item Let $K_3:= \frac{4\beta_{\sup}\dist^2(x_1,X^*)}{\alpha_{\inf}^2\underline{\lambda}(2-\overline{\lambda})\delta^2}$ and $\{v_k\}$ be given by \eqref{eq-dyn+stepsize}. Then
      \[\min\limits_{1\le k\le K_3} f(x_k)\le f^*+\delta^p+\epsilon.\]
\end{enumerate}
\end{theorem}
\begin{proof}
(i) We prove by contradiction, assuming for each $1\le k\le K_1$ that
\[
f(x_k)> f^*+\left(\frac{\beta_{\sup}}{\alpha_{\inf}} v+\delta\right)^p+\epsilon;
\]
hence, by (H1), we obtain by \eqref{eq-abs} (with $\mathbb{P}_{X^*}(x^k)$ in place of $x^*$) that
\begin{equation*}
\dist^2(x_{k+1},X^*)< \dist^2(x_k,X^*) - \alpha_{\inf} v \left(\frac{\beta_{\sup}}{\alpha_{\inf}} v+\delta\right) + \beta_{\sup} v^2
= \dist^2(x_k,X^*) - \alpha_{\inf} v \delta.
\end{equation*}
Summing the above inequality over $k = 1,\dots,K_1$, we obtain that
\[
\dist^2(x_{K_1+1},X^*)< \dist^2(x_1,X^*) - K_1 \alpha_{\inf} v \delta,
\]
which yields a contradiction with the definition of $K_1$. Assertion (i) of this theorem is proved.

(ii) Proving by contradiction, we assume that
\[
f(x_k)> f^*+\left(\frac{\beta_{\sup} }{\alpha_{\inf}}ck^{-s}+\delta\right)^p+\epsilon \quad \mbox{for each $1\le k\le K_2$}.
\]
Then, we obtain by \eqref{eq-abs} and \eqref{eq-dim+stepsize} that
\begin{equation*}
\dist^2(x_{k+1},X^*)< \dist^2(x_k,X^*) - \alpha_{\inf} v_k \left(\frac{\beta_{\sup}}{\alpha_{\inf}} v_k+\delta\right) + \beta_{\sup} v_k^2
= \dist^2(x_k,X^*) - \alpha_{\inf} c \delta k^{-s},
\end{equation*}
and thus,
\begin{equation*}
\begin{array}{lll}
\dist^2(x_{K_2+1},X^*)&< \dist^2(x_1,X^*) - \alpha_{\inf} c \delta \sum_{k=1}^{K_2} k^{-s}\\
&\le \dist^2(x_1,X^*) - \alpha_{\inf} c \delta \int_1^{K_2+1} t^{-s} {\rm d} t\\
&= \dist^2(x_1,X^*) - \alpha_{\inf} \frac{c \delta}{1-s} ((K_2+1)^{1-s}-1),
\end{array}
\end{equation*}
which is negative by the definition of $K_2$. This contradiction yields assertion (ii) of this theorem.

(iii) Proving by contradiction, we assume that
\[
f(x_k)> f^*+\delta^p+\epsilon\quad \mbox{for each $1\le k\le K_3$}.
\]
Then, it follows from \eqref{eq-abs} and \eqref{eq-dyn+stepsize} that
\begin{eqnarray}
\dist^2(x_{k+1},X^*)
&\le& \dist^2(x_k,X^*) -\frac{\alpha_k^2}{4\beta_k}\lambda_k(2-\lambda_k)(f(x_k)-f^*-\epsilon)^{\frac{2}{p}}\nonumber \\
&<& \dist^2(x_k,X^*) -\frac{\alpha_{\inf}^2}{4\beta_{\sup}}\underline{\lambda}(2-\overline{\lambda})\delta^2. \nonumber
\end{eqnarray}
Summing the above inequality over $k = 1,\dots,K_3$, we derive that
\[
\dist^2(x_{K_3+1},X^*)< \dist^2(x_1,X^*) - K_3 \frac{\alpha_{\inf}^2}{4\beta_{\sup}}\underline{\lambda}(2-\overline{\lambda})\delta^2,
\]
which yields a contradiction with the definition of $K_3$. The proof is complete.
\end{proof}
\begin{remark}\label{rem-ci}
Theorem \ref{thm-ce-abs} shows that the sequence satisfying conditions (H1) and (H2) possesses the computational complexity of $\mathcal{O}(1/k^{p})$, $\mathcal{O}(1/k^{p\min\{s,1-s\}})$ and $\mathcal{O}(1/k^{\frac{p}2})$ to fall within a certain region (expressed by an additive form of the stepsize and noise) of the optimal value when the constant, diminishing or dynamic stepsize rules are used, respectively. 
In particular, in the cases of the diminishing stepsize rule, the optimal complexity is gained when $s=\frac12$, and thus, $v_k=c k^{-\frac12}$ is the best choice among the type of \eqref{eq-dim+stepsize}.
\end{remark}

\subsection{Convergence rate analysis}
The establishment of convergence rates is significant in guaranteeing the numerical performance of relevant algorithms. The aim of this section is to establish the convergence rates for the sequence satisfying conditions ${\rm (H1)}$ and ${\rm (H2)}$ under the assumption of weak sharp minima and/or some suitable assumptions on the noise.

The concept of weak sharp minima was introduced by Burke and Ferris \cite{BurkeF93}, and has been extensively studied and widely used to analyze
the convergence rates of many optimization algorithms; see \cite{Bolte2017,HuSIOPT16,HuIP17,Anthony2017} and references therein. One natural extension of this concept is the weak sharp minima of H\"{o}lderian order; see \cite{HuCOA2019,HuEJOR15,StudWard99,Anthony2017} and references therein.
\begin{definition}\label{def-WSM-p}
Let $f:\R^n\to \R$, $X\subseteq \R^n$ and $X^*:=\arg\min \{f(x):x\in X\}$. Let $x^*\in X^*$, $S\subseteq \R^n$, $\eta>0$ and $q\ge 1$. $X^*$ is said to be
\begin{enumerate}[{\rm (a)}]
  \item the set of weak sharp minima of order $q$ for $f$ on $S$ over $X$ with modulus $\eta$ if
        \begin{equation*}\label{eq-WSM-p}
        f(y)-f(x^*)\ge \eta\, {\rm dist}^q(y,X^*)\quad {\rm for~each}~ y\in S\cap X;
        \end{equation*}
  \item the set of (global) weak sharp minima of order $q$ for $f$ over $X$ with modulus $\eta$ if $X^*$ is the set of weak sharp minima of order $q$ for $f$ on $\R^n$ over $X$ with modulus $\eta$;
  \item the set of boundedly weak sharp minima of order $q$ for $f$ over $X$ if, for each $r>0$ such that $X^*\cap \mathbf{B}(0, r)\neq \emptyset$, there exists $\eta_r>0$ such that $X^*$ is the set of weak sharp minima of order $q$ for $f$ on $\mathbf{B}(0, r)$ over $X$ with modulus $\eta_r$.
\end{enumerate}
\end{definition}


\begin{remark}\label{rem-wsm}
It is clear that the global weak sharp minima of order $q$ implies the boundedly weak sharp minima of order $q$. The larger the $q$, the less restrictive the global (resp. boundedly) weak sharp minima of order $q$. In particular, when $q=1$, this concept is reduced to the global (resp. boundedly) weak sharp minima; see, e.g., \cite{BurkeDeng02,BurkeF93}.
\end{remark}

The following theorems present the linear (or sublinear) convergence rates of the sequence $\{x_k\}$ satisfying ${\rm (H1)}$ and ${\rm (H2)}$ to a certain neighborhood of the optimal solution set when using different stepsize rules and under the assumption of boundedly weak sharp minima of H\"{o}lderian order. To this end, we further require the following condition to ensure the bounded property of $\{x_k\}$ when the constant or diminishing stepsize rule is adopted.
\begin{enumerate}
  \item[(H3)] For each $k\in \IN$,
  \begin{equation}\label{eq-abs-1}
    \|x_{k+1}-x_k\|\le \gamma_kv_k,
    \end{equation}
    where $\{\gamma_k\}$ is a sequence of positive scalars such that $\lim_{k\to \infty}\gamma_k=\gamma>0$.
\end{enumerate}
Condition (H3) characterizes an upper bound (related to the stepsize) of the distance between the successive two iterates, which is always satisfied for various subgradient methods for either convex or quasi-convex optimization problems; see, e.g., Section 4 and \cite{HuOpt18}.

\begin{theorem}\label{thm-abs-con}
Let $\{x_k\}\subseteq X$ satisfy ${\rm (H1)}$-${\rm (H3)}$ and $v_k\equiv v>0$. Suppose that $f$ is coercive and that $X^*$ is the set of boundedly weak sharp minima of order $q$ for $f$ over $X$ with modulus $\eta$. Then, the following assertions are true.
\begin{enumerate}[{\rm (i)}]
  \item If $q= 2p$, then either $x_k\in X^*_\epsilon$ for some $k\in \IN$ or there exist $\tau\in [0,1)$ and $N\in \IN$ such that
    \begin{equation*}\label{eq-cr-con2}
    \dist^2(x_{k+N},X^*)\le \tau^k \dist^2(x_N,X^*) +2^{\frac1p-1}\eta^{-\frac1{p}}\left(\epsilon^{\frac1p}+ \frac{\beta v}{\alpha}\right)\quad \mbox{for each } k\in \IN.
    \end{equation*}
  \item If $q>2p$ and $\epsilon^{\frac1p}+ \frac{\beta v}{\alpha}<\eta^{-\frac{2}{q-2p}}(\frac{2p}{\alpha vq})^{\frac{q}{q-2p}}$, then either $x_k\in X^*_\epsilon$ for some $k\in \IN$ or there exist $\tau\in (0,1)$ and $N\in \IN$ such that
    \begin{equation*}\label{eq-cr-con2a}
    \dist^2(x_{k+N},X^*)\le \tau^k \dist^2(x_N,X^*) +2^{\frac2q-\frac{2p}q}{\eta}^{-\frac{2}q}\left(\epsilon^{\frac1p}+ \frac{\beta v}{\alpha}\right)^{\frac{2p}q}\quad \mbox{for each } k\in \IN.
    \end{equation*}
\end{enumerate}
\end{theorem}
\begin{proof}
Without loss of generality, we assume that $f(x_k)> f^*+\epsilon$ for each $k\in \IN$; otherwise, this theorem holds automatically. Consequently,  (H1) says that, for each $x^* \in X^*$ and $k\in \IN$,
\begin{equation}\label{eq-pro-con6a}
\|x_{k+1}-x^*\|^2 \le  \|x_k-x^*\|^2  -\alpha_kv (f(x_k)-f^*-\epsilon)^{\frac1p}+ \beta_kv^2.
\end{equation}

We first claim that $\{x_k\}$ is bounded. To this end, we fix $\kappa>1$ and define
\begin{equation}\label{eq-pro-con6}
\sigma:=\kappa\left(\frac{\beta v}{\alpha}\right)^p+\epsilon,\quad X^*_\sigma:=X \cap \mbox{lev}_{\le f^*+\sigma}f \quad \mbox{and}\quad \rho(\sigma):=\max \limits_{x \in X^*_\sigma}\dist (x,X^*).
\end{equation}
By the coercive assumption, it follows that its sublevel set $\mbox{lev}_{\le f^*+\sigma}f$ is bounded, and so is $X^*_\sigma$. Then, one has by \eqref{eq-pro-con6} that $\rho(\sigma)<\infty$. By ${\rm (H2)}$ and ${\rm (H3)}$, there exists $N\in \IN$ such that
\begin{equation}\label{eq-pro-con1}
\alpha\kappa^{-\frac1{2p}}<\alpha_k <\alpha\kappa^{\frac1{2p}}, \quad \beta\kappa^{-\frac1{2p}}<\beta_k < \beta\kappa^{\frac1{2p}}  \quad  \mbox{and} \quad  \gamma_k < 2\gamma \quad \mbox{for each } k\ge N.
\end{equation}
Fix $k\ge N$. Below, we show
\begin{equation}\label{eq-pro-con3}
\dist(x_{k+1},X^*) < \max \{\dist(x_k,X^*),\rho(\sigma)+2\gamma v\}
\end{equation}
by claiming the following two implications:
\begin{equation}\label{eq-pro-con4}
[f(x_k)> f^*+\sigma] \quad \Rightarrow \quad [\dist(x_{k+1},X^*)<\dist(x_k,X^*)];
\end{equation}
\begin{equation}\label{eq-pro-con5}
[f(x_k)\le f^*+\sigma] \quad \Rightarrow \quad [\dist(x_{k+1},X^*)<\rho(\sigma)+2\gamma v].
\end{equation}
To prove \eqref{eq-pro-con4}, we suppose that $f(x_k)> f^*+\sigma$. Then, for each $x^* \in X^*$, one has by \eqref{eq-pro-con6a} that
\[
\| x_{k+1}-x^* \|^2 < \| x_k-x^* \|^2-\alpha_k v (\sigma-\epsilon)^{\frac1p}+\beta_k v^2 < \| x_k -x^* \|^2
\]
(due to \eqref{eq-pro-con6} and \eqref{eq-pro-con1}). Consequently, \eqref{eq-pro-con4} can be proved by selecting $x^* =\mathbb{P}_{X^*}(x_k)$.
To show \eqref{eq-pro-con5}, we assume that $f(x_k)\le f^*+\sigma$. Then $x_k \in X^*_\sigma$, and so \eqref{eq-pro-con6} says that $\dist(x_k,X^*) \le \rho(\sigma)$. This, together with \eqref{eq-abs-1} and \eqref{eq-pro-con1}, shows \eqref{eq-pro-con5}. Therefore, \eqref{eq-pro-con3} is proved, as desired.

Note by Theorem \ref{thm-abs}(i) that $\liminf_{k \to \infty} f(x_k)\le f^*+(\frac{\beta v}{\alpha})^p+\epsilon$, and note by $\kappa>1$ and \eqref{eq-pro-con6} that $\sigma>(\frac{\beta v}{\alpha})^p+\epsilon$. Then, we can assume, without loss of generality, that $f(x_{N})\le f^* + \sigma$ (otherwise, we can choose a larger $N$). Consequently, we have by \eqref{eq-pro-con5} that $\dist(x_{N+1},X^*)<\rho(\sigma)+2\gamma v$, and inductively obtain by \eqref{eq-pro-con3} that
\begin{equation}\label{eq-pro-con4ab}
\dist(x_k,X^*)<\rho(\sigma)+2\gamma v \quad \mbox{for each } k> N.
\end{equation}
Hence, $\{x_k\}$ is proved to be bounded (since $X^*\subseteq X^*_{\sigma}$ is bounded), as desired.
That is, there exists $r>0$ such that $X^*\subseteq \mathbf{B}(0,r)$ and $\{x_k\}\subseteq \mathbf{B}(0,r)$ for each $k\in \IN$. Then, by assumption of boundedly weak sharp minima property of order $q$, there exists $\eta>0$ such that
\begin{equation}\label{eq-pro-con7}
f(x_k)-f^* \ge \eta\, \dist^{q}(x_k,X^*)\quad \mbox{for each } k\in \IN.
\end{equation}
Selecting $x^* =\mathbb{P}_{X^*}(x_k)$, we deduce by \eqref{eq-pro-con6a} and \eqref{eq-pro-con1} that, for each $k\ge N$,
\begin{equation}\label{eq-pro-con6b}
\begin{array}{lll}
&\dist^2(x_{k+1},X^*)\\
&\le \dist^2(x_k,X^*)-\alpha v\kappa^{-\frac1{2p}} (f(x_k)-f^*-\epsilon)^{\frac1p}+ \beta\kappa^{\frac1{2p}} v^2\\
&\le \dist^2(x_k,X^*)-2^{1-\frac1p} \alpha v\kappa^{-\frac1{2p}} (f(x_k)-f^*)^{\frac1p}+\alpha v\kappa^{-\frac1{2p}}\epsilon^{\frac1p}+ \beta\kappa^{\frac1{2p}} v^2,
\end{array}
\end{equation}
where the last inequality holds because
\begin{equation}\label{eq-lpnorm}
(a-b)^\gamma\ge 2^{1-\gamma}a^\gamma-b^\gamma \quad \mbox{whenever } a\ge b\ge 0 \mbox{ and  } \gamma\ge 1
\end{equation}
(cf. \cite[Lemma 4.1]{HuangYang03}).
Below, we prove this theorem in the following two cases.

(i) Suppose that $q= 2p$. Setting $\tau:= (1-2^{1-\frac1p}\alpha v\kappa^{-\frac1{2p}} \eta^{\frac1p})_+\in [0,1)$ and substituting \eqref{eq-pro-con7} into \eqref{eq-pro-con6b}, we achieve that
\[
\dist^2(x_{k+1},X^*)\le \tau \dist^2(x_k,X^*)+\alpha v\kappa^{-\frac1{2p}}\epsilon^{\frac1p}+ \beta\kappa^{\frac1{2p}} v^2\quad \mbox{for each }  k\ge N.
\]
Then, we inductively obtain that
\[
\dist^2(x_{k+N},X^*)\le \tau^k \dist^2(x_N,X^*) +\frac1{1-\tau}(\alpha v\kappa^{-\frac1{2p}}\epsilon^{\frac1p}+ \beta\kappa^{\frac1{2p}} v^2)\quad \mbox{for each } k\in \IN;
\]
consequently, the conclusion follows (noting that $\kappa>1$ is arbitrary).

(ii) Suppose that $q> 2p$ and $\alpha \epsilon^{\frac1p}+\beta v<(\alpha^p \eta)^{-\frac{2}{q-2p}}(v\frac{q}{2p})^{-\frac{q}{q-2p}}$. We obtain by \eqref{eq-pro-con7} and \eqref{eq-pro-con6b} that, for each  $k\ge N$,
\[
\dist^2(x_{k+1},X^*)\le \dist^2(x_k,X^*)-2^{1-\frac1p}\alpha v\kappa^{-\frac1{2p}} \eta^{\frac1p}\dist^{\frac{q}p}(x_k,X^*)+\alpha v\kappa^{-\frac1{2p}}\epsilon^{\frac1p}+ \beta\kappa^{\frac1{2p}} v^2.
\]
Then, Lemma \ref{lem-new}(ii) is applicable (with $\dist^2(x_k,X^*)$, $\frac{q}{2p}-1$, $2^{1-\frac1p}\alpha v\kappa^{-\frac1{2p}} \eta^{\frac1p}$, $\alpha v\kappa^{-\frac1{2p}}\epsilon^{\frac1p}+ \beta\kappa^{\frac1{2p}} v^2$ in place of $u_k$, $r$, $a$, $b$) to obtaining the conclusion.
\end{proof}

\begin{theorem}\label{thm-abs-dyn}
Let $\{x_k\}\subseteq X$ satisfy ${\rm (H1)}$ and ${\rm (H2)}$, and $\{v_k\}$ be given by \eqref{eq-dyn+stepsize}.  Suppose that $X^*$ is the set of boundedly weak sharp minima of order $q$ for $f$ over $X$ with modulus $\eta$. Then, the following assertions are true.
\begin{enumerate}[{\rm (i)}]
  \item If $q= p$ and $\epsilon\ge 0$, then either $x_k\in X^*_\epsilon$ for some $k\in \IN$ or there exist $\tau\in [0,1)$ and $N\in \IN$ such that
    \begin{equation}\label{eq-cr-dyn2}
    \dist^2(x_{k+N},X^*)\le \tau^k \dist^2(x_N,X^*) +2^{\frac2p-1}\eta^{-\frac2{p}}\epsilon^{\frac2p} \quad \mbox{for each }  k\in \IN.
    \end{equation}
    \item If $q>p$ and $\epsilon= 0$, then either $x_k\in X^*$ for some $k\in \IN$ or there exist $\gamma>0$ and $N\in \IN$ such that
      \begin{equation}\label{eq-cr-dyn2b}
      \dist^2(x_{k+N},X^*)\le \frac{\dist^2(x_N,X^*)}{(1+\gamma k)^{\frac{p}{q-p}}}\quad \mbox{for each } k\in \IN.
      \end{equation}
  \item If $q>p$ and $0<\epsilon<\left(\frac{\alpha^2q}{4\beta p}\underline{\lambda}(2-\overline{\lambda})\right)^{-\frac{pq}{2(q-p)}}\eta^{-\frac{p}{q-p}}$, then either $x_k\in X^*_\epsilon$ for some $k\in \IN$ or there exist $\tau\in (0,1)$ and $N\in \IN$ such that
    \begin{equation}\label{eq-cr-dyn2a}
    \dist^2(x_{k+N},X^*)\le \tau^k \dist^2(x_N,X^*) +2^{\frac2q-\frac{p}q}\eta^{-\frac2q} \epsilon^{\frac{2}q}\quad \mbox{for each } k\in \IN.
    \end{equation}
\end{enumerate}
\end{theorem}
\begin{proof}
Without loss of generality, we assume that $f(x_k)> f^*+\epsilon$ for each $k\in \IN$; otherwise, this theorem holds automatically.
Combining \eqref{eq-abs} with \eqref{eq-dyn+stepsize}, we obtain that there exists $N\in \IN$ such that, for each $k\ge N$,
\begin{eqnarray}
& & \dist^2(x_{k+1},X^*)-\dist^2(x_{k},X^*) \nonumber \\
& &\le -\frac{\alpha_k^2}{4\beta_k}\lambda_k(2-\lambda_k)(f(x_k)-f^*-\epsilon)^{\frac2p} \label{eq-dyn-ep} \\
& &\le -2^{1-\frac2p}\kappa^{-\frac3{2p}}\frac{\alpha^2}{4\beta}\underline{\lambda}(2-\overline{\lambda})(f(x_k)-f^*)^{\frac2p}
+\kappa^{-\frac3{2p}}\frac{\alpha^2}{4\beta}\underline{\lambda}(2-\overline{\lambda})\epsilon^{\frac2p} \nonumber
\end{eqnarray}
(due to  \eqref{eq-pro-con1} and \eqref{eq-lpnorm}). It follows from the proof of Theorem \ref{thm-abs}(iii) (cf. \eqref{eq-abs-dyn}) that $\{x_k\}$ is bounded. Then, there exists $r>0$ such that $X^*\cap \mathbf{B}(0,r)\neq \emptyset$ and $\{x_k\}\subseteq \mathbf{B}(0,r)$ for each $k\in \IN$. By the assumption of boundedly weak sharp minima property of order $q$, there exists $\eta>0$ such that
\begin{equation}\label{eq-dyn-wsm1}
f(x_k)-f^*\ge \eta\,\dist^q(x_k,X^*)\quad \mbox{for each } k\in \IN.
\end{equation}

(i) Suppose that $q= p$. Setting $\tau:= (1-2^{1-\frac2p}\kappa^{-\frac3{2p}}\frac{\alpha^2}{4\beta}\underline{\lambda}(2-\overline{\lambda})\eta^{\frac2p})_+\in [0,1)$, we obtain by substituting \eqref{eq-dyn-wsm1} into \eqref{eq-dyn-ep} that
\[
\dist^2(x_{k+1},X^*)\le \tau\dist^2 (x_k,X^*)+\kappa^{-\frac3{2p}}\frac{\alpha^2}{4\beta}\underline{\lambda}(2-\overline{\lambda})\epsilon^{\frac2p}\quad \mbox{for each } k\ge N.
\]
Then, we inductively obtain
\[
\dist^2(x_{k+N},X^*)\le \tau^k \dist^2(x_N,X^*) +\frac1{1-\tau}\kappa^{-\frac3{2p}}\frac{\alpha^2}{4\beta}\underline{\lambda}(2-\overline{\lambda})\epsilon^{\frac2p} \quad \mbox{for each } k\in \IN;
\]
hence, \eqref{eq-cr-dyn2} holds (noting that $\kappa>1$ is arbitrary), and assertion (i) is proved.

(ii) Suppose that $q>p$ and $\epsilon= 0$. Then, combining \eqref{eq-dyn-ep} and  \eqref{eq-dyn-wsm1} implies that, for each $k\ge N$,
\[
\dist^2(x_{k+1},X^*)\le \dist^2(x_k,X^*)
-2^{1-\frac2p}\kappa^{-\frac3{2p}}\frac{\alpha^2}{4\beta}\underline{\lambda}(2-\overline{\lambda})\eta^{\frac2p}\dist^{\frac{2q}p}(x_k,X^*).
\]
Then, Lemma \ref{lem-new}(i) is applicable (with $\dist^2(x_k,X^*)$, $2^{1-\frac2p}\kappa^{-\frac3{2p}}\frac{\alpha^2}{4\beta}\underline{\lambda}(2-\overline{\lambda})\eta^{\frac2p}$, $\frac{q}p-1$ in place of $u_k$, $a$, $r$) to concluding that \eqref{eq-cr-dyn2b} holds
with $\gamma:=2^{1-\frac2p}\kappa^{-\frac3{2p}}\frac{\alpha^2(q-p)}{4\beta p}\underline{\lambda}(2-\overline{\lambda})\eta^{\frac2p}
\dist^{\frac{2q}{p}-2}(x_N,X^*)$.

(iii) Suppose that $q>p$ and $0<\epsilon<\left(\frac{\alpha^2q}{4\beta p}\underline{\lambda}(2-\overline{\lambda})\right)^{-\frac{pq}{2(q-p)}}\eta^{-\frac{p}{q-p}}$. We obtain by \eqref{eq-dyn-ep} and \eqref{eq-dyn-wsm1} that, for each  $k\ge N$,
\[
\dist^2(x_{k+1},X^*)\le \dist^2(x_k,X^*)-2^{1-\frac2p}\kappa^{-\frac3{2p}}\frac{\alpha^2}{4\beta}\underline{\lambda}(2-\overline{\lambda})\eta^{\frac2p}
\dist^{\frac{2q}p}(x_k,X^*)+\kappa^{-\frac3{2p}}\frac{\alpha^2}{4\beta}\underline{\lambda}(2-\overline{\lambda})\epsilon^{\frac2p}.
\]
Then Lemma \ref{lem-new}(ii) is applicable (with $\dist^2(x_k,X^*)$, $\frac{q}{p}-1$, $2^{1-\frac2p}\kappa^{-\frac3{2p}}\frac{\alpha^2}{4\beta}\underline{\lambda}(2-\overline{\lambda})\eta^{\frac2p}$, $\kappa^{-\frac3{2p}}\frac{\alpha^2}{4\beta}\underline{\lambda}(2-\overline{\lambda})\epsilon^{\frac2p}$ in place of $u_k$, $r$, $a$, $b$) to concluding that there exists $\tau\in (0,1)$ such that \eqref{eq-cr-dyn2a} is satisfied.
\end{proof}

\begin{theorem}\label{thm-abs-dimi}
Let $\{x_k\}\subseteq X$ satisfy ${\rm (H1)}$-${\rm (H3)}$, and $\{v_k\}$ be given by \eqref{eq-dim+stepsize}. Suppose that $f$ is coercive and that $X^*$ is the set of boundedly weak sharp minima of order $2p$ for $f$ over $X$ with modulus $\eta$. Then, the following assertions are true.
\begin{enumerate}[{\rm (i)}]
  \item If $\epsilon=0$, then either $x_k\in X^*$ for some $k\in \IN$ or there exists $N\in \IN$ such that
    \begin{equation*}\label{eq-cr-con2}
    \dist^2(x_k,X^*)\le \frac{\beta c}{\alpha }\left(\frac{2}{\eta}\right)^{\frac1p} k^{-s}\quad \mbox{for each } k\ge N.
    \end{equation*}
  \item If $\epsilon>0$, then either $x_k\in X^*_\epsilon$ for some $k\in \IN$ or there exist $C>0$ and $\tau\in (0,1)$ such that
  \begin{equation*}
    \dist^2(x_k,X^*)\le C\tau^k +\left(\frac{2\epsilon}{\eta}\right)^{\frac1p}\quad \mbox{for each } k\in \IN.
    \end{equation*}
\end{enumerate}
\end{theorem}
\begin{proof}
Without loss of generality, we assume that $f(x_k)> f^*+\epsilon$ for each $k\in \IN$; otherwise, this theorem holds automatically. Then,  (H1) and \eqref{eq-dim+stepsize} show that, for each $x^* \in X^*_\epsilon$ and $k\in \IN$,
\begin{equation}\label{eq-pro-dimi}
\|x_{k+1}-x^*\|^2 \le  \|x_k-x^*\|^2  -\alpha_kck^{-s} (f(x_k)-f^*-\epsilon)^{\frac1p}+ \beta_kc^2k^{-2s}.
\end{equation}
Fix $\kappa>1$ and $\sigma>\epsilon$, and define $X^*_\sigma$ and $\rho(\sigma)$ by \eqref{eq-pro-con6}. By ${\rm (H2)}$ and ${\rm (H3)}$, there exists $N\ge (\frac{c\beta}{\alpha})^{\frac1s}(\frac{\kappa}{\sigma-\epsilon})^{\frac1{sp}}$ such that \eqref{eq-pro-con1} is satisfied. Similar to the arguments that we did for \eqref{eq-pro-con4ab}, we can derive that
$\dist(x_k,X^*)<\rho(\sigma)+2\gamma c$ for each $k> N$.
Hence, $\{x_k\}$ is bounded, and then, by assumption of boundedly weak sharp minima property of order $2p$, there exists $\eta>0$ such that \eqref{eq-pro-con7} is satisfied with $q=2p$. This, together with \eqref{eq-pro-dimi} and \eqref{eq-lpnorm}, implies that, for each $k\ge N$,
\begin{equation}\label{eq-pro-dimi2}
\dist^2(x_{k+1},X^*)\le (1-2^{1-\frac1p}\alpha \kappa^{-\frac1{2p}}\eta^{\frac1p} ck^{-s})\dist^2(x_k,X^*)+\alpha \kappa^{-\frac1{2p}}\epsilon^{\frac1p}ck^{-s}+ \beta\kappa^{\frac1{2p}} c^2k^{-2s}.
\end{equation}

(i) Suppose that $\epsilon=0$. Lemma \ref{lem-dimi}(i) is applicable (with $2^{1-\frac1p}\alpha \kappa^{-\frac1{2p}}\eta^{\frac1p} c$, $\beta\kappa^{\frac1{2p}} c^2$, $2s$ in place of $a$, $b$, $t$) to obtaining the conclusion (as $\kappa>1$ is arbitrary). 

(ii) Suppose that $\epsilon>0$. Letting $N\ge (\frac{c\beta}{\alpha})^{\frac1s}(\frac{\kappa}{\epsilon})^{\frac1{sp}}$, \eqref{eq-pro-dimi2} is reduced to
\[
\dist^2(x_{k+1},X^*)\le (1-2^{1-\frac1p}\alpha \kappa^{-\frac1{2p}}\eta^{\frac1p} ck^{-s})\dist^2(x_k,X^*)+2\alpha \kappa^{-\frac1{2p}}\epsilon^{\frac1p}ck^{-s},
\]
and then, Lemma \ref{lem-dimi}(ii) is applicable (with $2^{1-\frac1p}\alpha \kappa^{-\frac1{2p}}\eta^{\frac1p} c$ and $2\alpha \kappa^{-\frac1{2p}}\epsilon^{\frac1p}c$ in place of $a$ and $b$) to obtaining the conclusion.
\end{proof}

\begin{remark}
Theorems on convergence rates improve the results in \cite{HuOpt18}, in which only the global convergence theorems were provided.

{\rm (i)} Theorems \ref{thm-abs-con} and \ref{thm-abs-dyn} show the linear convergence rates of the sequence satisfying ${\rm (H1)}$-${\rm (H3)}$ to a certain region (i.e., $\mathcal{O}((\epsilon^{\frac1p}+v)^{\frac{2p}q})$ or $\mathcal{O}(\epsilon^{\frac{2}q})$) of the optimal solution set under the boundedly weak sharp minima of order $q$ when using the constant or dynamic stepsize rules, respectively.

{\rm (ii)} In the special case when $\epsilon=0$ and using the dynamic stepsize rule, Theorem \ref{thm-abs-dyn} presents the linear (or sublinear) convergence of the sequence satisfying ${\rm (H1)}$-${\rm (H2)}$ to the optimal solution set under the boundedly weak sharp minima of order $q$.

{\rm (iii)} When using the diminishing stepsize rule \eqref{eq-dim+stepsize} and under the boundedly weak sharp minima of order $2p$, Theorem \ref{thm-abs-dimi} shows the sublinear convergence rate to the optimal solution set or the linear convergence rate to a certain region (i.e., $\mathcal{O}(\epsilon^{\frac1p})$) of the optimal solution set for the exact or inexact framework, respectively.
\end{remark}

\section{Applications to subgradient methods}
Quasi-convex optimization plays an important role in various fields such as economics, finance and engineering.
The subgradient method is a popular algorithm for solving constrained quasi-convex optimization problems. It was shown in \cite{HuOpt18} that several types of subgradient methods for solving quasi-convex optimization problem satisfy conditions (H1)-(H3) assumed in the unified framework (with $\epsilon=0$) under the H\"{o}lder condition. Hence, in this section, we directly apply the convergence theorems obtained in the preceding section to establish the convergence properties of several subgradient methods for solving quasi-convex optimization problems. The convergence theorems (resp. Theorems \ref{thm-sgm}, \ref{thm-sgmin} and \ref{thm-sgmcg}) cover the existing results in the literature of subgradient methods (resp. \cite[Theorem 4.2]{HuOpt18}, \cite[Theorem 3.1]{HuEJOR15} and \cite[Theorems 3.3 and 3.5]{HuJNCA16b}). To the best of our knowledge, the theorems of complexity estimation (i.e., Theorems \ref{thm-cp-sgm}, \ref{thm-cp-sgmin} and \ref{thm-cp-sgmcg}) and convergence rates (i.e., Theorems \ref{thm-cr-sgm}, \ref{thm-cr-sgmin} and \ref{thm-cr-sgmcg}) of subgradient methods for quasi-convex optimization are new in the literature. Throughout the whole section, we make the following blanket assumption:
\begin{itemize}
  \item $f:\R^n \to \R$ is quasi-convex and continuous, and satisfies the H\"{o}lder condition of order $p$ with modulus $L$ on $\R^n$.
\end{itemize}
\subsection{Subgradient method}
It was claimed in \cite{HuOpt18} the standard subgradient method (i.e., Algorithm \ref{alg-SGM}) satisfies the following basic inequality under the blanket assumption:
\[
\| x_{k+1}-x^* \|^2 \le \| x_k-x^* \|^2 - 2v_k \left(\frac{f(x_k)-f^*}{L}\right)^\frac1p + {v_k}^2
\]
whenever $f(x_k)>f^*$, and $\| x_{k+1}-x_k\|\le v_k$. That is, conditions (H1)-(H3) are satisfied with
\[
\epsilon=0, \quad \alpha_k\equiv 2L^{-\frac1p}, \quad \beta_k\equiv 1, \quad \gamma_k\equiv 1.
\]
Therefore, the convergence theorems established in the preceding section can be directly applied to establish the convergence properties of the standard subgradient method as follows.
\begin{theorem}\label{thm-sgm}
Let $\{x_k\}$ be a sequence generated by Algorithm \ref{alg-SGM}.
\begin{enumerate}[{\rm (i)}]
  \item If $v_k\equiv v>0$, then $\liminf_{k\to \infty} f(x_k)\le f^*+L\left(\frac12 v\right)^p$.
  \item If $\{v_k\}$ is given by \eqref{eq-dim+stepsize}, then $\liminf_{k\to \infty} f(x_k)\le f^*$.
  \item If $\{v_k\}$ is given by
  \begin{equation}\label{eq-dyn-sgm}
    v_k:= \lambda_k \left(\frac{f(x_k)-f^*}{L}\right)^{\frac1p},\quad \mbox{where $0<\underline{\lambda}\le \lambda_k \le \overline{\lambda}<2$},
  \end{equation}
  then $\lim_{k\to \infty} f(x_k)= f^*$.
\end{enumerate}
\end{theorem}

\begin{remark}
For the dynamic stepsize rule \eqref{eq-dyn-sgm} (\eqref{eq-dyn-sgmin} or \eqref{eq-dyn-sgmcg} in the sequel), once $x_k$ enters $X^*$ (or $X^*_\epsilon$), the stepsize will be zero, the iterates will stay at the optimal solution (or the approximate optimal solution), and thus, the conclusions of global convergence and convergence rate follow automatically.
\end{remark}

\begin{theorem}\label{thm-cp-sgm}
Let $\delta>0$, and let $\{x_k\}$ be a sequence generated by Algorithm \ref{alg-SGM}.
\begin{enumerate}[{\rm (i)}]
  \item Let $K_1:= \frac{L^{\frac1p}\dist^2(x_1,X^*)}{2 v \delta} $ and $v_k\equiv v>0$. Then
  $\min\limits_{1\le k\le K_1} f(x_k)\le f^*+\left(\frac12 L^{\frac1p}v+\delta\right)^p$.
  \item Let $K_2:=\left(\frac{(1-s)L^{\frac1p}\dist^2(x_1,X^*)}{2 c \delta}\right)^{\frac1{1-s}}$ and $\{v_k\}$ be given by \eqref{eq-dim+stepsize}. Then
  $\min\limits_{1\le k\le K_2} f(x_k)\le f^*+\left(\frac12 L^{\frac1p}ck^{-s}+\delta\right)^p$.
  \item Let $K_3:= \frac{L^{\frac2p}\dist^2(x_1,X^*)}{\underline{\lambda}(2-\overline{\lambda})\delta^2} $ and $\{v_k\}$ be given by \eqref{eq-dyn-sgm}. Then $\min\limits_{1\le k\le K_3} f(x_k)\le f^*+\delta^p$.
\end{enumerate}
\end{theorem}

\begin{theorem}\label{thm-cr-sgm}
Let $\{x_k\}$ be a sequence generated by Algorithm \ref{alg-SGM}.  Suppose that $X^*$ is the set of boundedly weak sharp minima of order $q$ for $f$ over $X$ with modulus $\eta$.
\begin{enumerate}[{\rm (I)}]
  \item Suppose that $v_k\equiv v>0$ and that $f$ is coercive.
      \begin{enumerate}[{\rm (i)}]
      \item If $q= 2p$, then either $x_k\in X^*$ for some $k\in \IN$ or there exist $\tau\in [0,1)$ and $N\in \IN$ such that
        \[
        \dist^2(x_{k+N},X^*)\le \tau^k \dist^2(x_N,X^*) +2^{\frac1p-2}\left(\frac{L}{\eta}\right)^{\frac1p}v \quad \mbox{for each } k\in \IN.
        \]
      \item If $q>2p$ and $v<\left(2^{-p}\frac{L}{\eta}(\frac{2p}{q})^{\frac{q}{2}}\right)^{\frac{1}{q-p}}$, then either $x_k\in X^*$ for some $k\in \IN$ or there exist $\tau\in (0,1)$ and $N\in \IN$ such that
        \[
        \dist^2(x_{k+N},X^*)\le \tau^k \dist^2(x_N,X^*) +2^{\frac2q-\frac{4p}q}\left(\frac{L}{\eta}\right)^{\frac2q}v^{\frac{2p}q}\quad \mbox{for each } k\in \IN.
        \]
    \end{enumerate}
  \item Suppose that $\{v_k\}$ is given by \eqref{eq-dim+stepsize}, $q=2p$ and that $f$ is coercive. Then, either $x_k\in X^*$ for some $k\in \IN$ or there exists $N\in \IN$ such that
    \begin{equation*}\label{eq-cr-con2}
    \dist^2(x_k,X^*)\le \frac{c}2\left(\frac{2L}\eta\right)^{\frac1p}k^{-s}\quad \mbox{for each } k\ge N.
    \end{equation*}
  \item Suppose that $\{v_k\}$ is given by \eqref{eq-dyn-sgm}.
    \begin{enumerate}
      \item[{\rm (i)}] If $q= p$, then there exist $\tau\in [0,1)$ and $N\in \IN$ such that
        \[
        \dist^2(x_{k+N},X^*)\le \tau^k \dist^2(x_N,X^*) \quad \mbox{for each }  k\in \IN.
        \]
        \item[{\rm (ii)}] If $q>p$, then there exist $\gamma>0$ and $N\in \IN$ such that
          \[
          \dist^2(x_{k+N},X^*)\le \frac{\dist^2(x_N,X^*)}{(1+\gamma k)^{\frac{p}{q-p}}}\quad \mbox{for each } k\in \IN.
          \]
    \end{enumerate}
\end{enumerate}
\end{theorem}

\begin{remark}
(i) When using the dynamic stepsize rule, the linear convergence rate of the subgradient method for solving convex optimization problems was shown in \cite[Theorem 2.5]{Brannlund1995} under the assumption of weak sharp minima. Theorem \ref{thm-cr-sgm} remarkably extends the result in \cite{Brannlund1995} to quasi-convex optimization, using several typical stepsize rules and under the weaker assumption of weak sharp minima of H\"{o}lderian order.

(ii) When using the constant stepsize rule, \cite[Theorem 1]{Johnstone2019} proved the linear convergence rate of the subgradient method for solving convex optimization problems under the assumption of weak sharp minima of H\"{o}lderian order. Theorem \ref{thm-cr-sgm} extends the result in \cite{Brannlund1995} to quasi-convex optimization and using several typical stepsize rules.
\end{remark}
\subsection{Inexact subgradient method}
In many applications, the computation error stems from practical considerations and is inevitable in the computing process. To meet the requirement of practical applications, an inexact subgradient method was proposed in \cite{HuEJOR15} to solve a constrained quasi-convex optimization problem \eqref{eq-QCO}, in which an $\epsilon$-quasi-subgradient is employed (with an additional noise), and the global convergence theorem is established there. This section is devoted to establishing the iteration complexity and convergence rates of the inexact subgradient method for solving problem \eqref{eq-QCO}, which is formally described as follows.
\begin{algorithm}\label{alg-IneSGM}
{\rm
Let $\epsilon>0$. Select an initial point $x_1\in \R^n$ and a sequence of stepsizes $\{v_k\}\subseteq (0,+\infty)$. For each $k\in \IN$, having $x_k$, we select
$g_k\in \partial^*_\epsilon f(x_k)\cap \mathbf{S}$ and update $x_{k+1}$ by
\begin{equation}\label{eq-ineSGM}
x_{k+1}:=\mathbb{P}_X (x_k-v_k g_k).
\end{equation}
}
\end{algorithm}

We claim that the inexact subgradient method satisfies conditions (H1)-(H3) under the blanket assumption. To this end, we provide in the following lemma an important property of a quasi-convex function, which is inspired by \cite[Proposition 2.1]{Konnov03}.

\begin{lemma}\label{lem-HC-epsilon}
Let $x\in X$ be such that $f(x)>f^*+\epsilon$ and let $g\in \partial^*_{\epsilon} f(x)\cap \mathbf{S}$. Then, it holds that
\[
\langle g,x-x^*\rangle \ge \left(\frac{f(x)-f^*-\epsilon}{L}\right)^{\frac1p} \quad \mbox{for each } x^*\in X^*.
\]
\end{lemma}
\begin{proof}
Fix $x\in X$ be such that $f(x)>f^*+\epsilon$. Then, the level set ${\rm lev}_{<f(x)-\epsilon}f$ is nonempty open and convex by the blanket assumption that $f$ is quasi-convex and continuous on $\R^n$. Given $x^*\in X^*$, we define
\begin{equation}\label{eq-r}
r:=\inf\left\{\|y-x^*\|:y\in {\rm bd}\left({\rm lev}_{<f(x)-\epsilon}f\right)\right\},
\end{equation}
where ${\rm bd}(Z)$ denotes the boundary of the set $Z$. It follows that $r>0$ by the fact that $f(x)-\epsilon>f(x^*)$ and the H\"{o}lder condition. Furthermore, we have by Definition \ref{def-HC} that
\[f(y)-f^*\le L \dist^p(y,X^*)\quad \mbox{for each } y\in \R^n.\]
Taking the infimun over ${\rm bd}\left({\rm lev}_{<f(x)-\epsilon}f\right)$, we can show that
\begin{equation}\label{eq-HC-1}
f(x)-f^*-\epsilon\le L \inf\left\{ \dist^p(y,X^*):y\in {\rm bd}\left({\rm lev}_{<f(x)-\epsilon}f\right)\right\}\le Lr^p.
\end{equation}
Let $\delta\in (0,1)$. Since $g\in \partial^*_{\epsilon} f(x)\cap \mathbf{S}$, we obtain by \eqref{eq-r} that $x^*+\delta r g\in {\rm lev}_{<f(x)-\epsilon}f$. Hence, it follows by definition that
\[\langle g,x-x^*\rangle=\langle g,x-(x^*+\delta r g)\rangle+\delta r  \ge \delta r.\]
Since $\delta\in (0,1)$ is arbitrary, one has $\langle g,x-x^*\rangle \ge r$. This, together with \eqref{eq-HC-1}, implies the conclusion.
\end{proof}

\begin{lemma}\label{lem-BI-epsilon}
Let $\{x_k\}$ be a sequence generated by Algorithm \ref{alg-IneSGM}.
\begin{enumerate}[{\rm (i)}]
  \item $\|x_{k+1}-x_k\|\le v_k$.
  \item   If $f(x_k)>f^*+\epsilon$, then it holds for each $x^*\in X^*$ that
    \begin{equation*}\label{eq-BI-epsilon}
    \|x_{k+1}-x^*\|^2\le \|x_k-x^*\|^2-2v_k \left(\frac{f(x_k)-f^*-\epsilon}{L}\right)^{\frac{1}{p}}+ v_k^2.
    \end{equation*}
\end{enumerate}
\end{lemma}
\begin{proof}
Assertion (i) of this theorem directly follows from Algorithm \ref{alg-IneSGM} (cf. \eqref{eq-ineSGM}). To show assertion (ii), we suppose that $f(x_k)>f^*+\epsilon$ and fix $x^*\in X^*$. In view of Algorithm \ref{alg-IneSGM}, it follows from the nonexpansive property of projection operator that
\begin{equation}\label{eq-BI-0}
\begin{array}{lll}
\|x_{k+1}-x^*\|^2&\le \left\|x_k-v_k g_k-x^*\right\|^2 \\
&= \|x_k-x^*\|^2-2v_k \left\langle g_k, x_k-x^*\right\rangle+ v_k^2.
\end{array}
\end{equation}
By the assumption that $f(x_k)>f^*+\epsilon$, Lemma \ref{lem-HC-epsilon} is applicable to concluding that
\begin{equation*}\label{eq-BI-0a}
\langle g_k, x_k-x^*\rangle \ge \left(\frac{f(x_k)-f^*-\epsilon}{L}\right)^{\frac1p}.
\end{equation*}
This, together with \eqref{eq-BI-0}, implies the conclusion. 
\end{proof}

Lemma \ref{lem-BI-epsilon} shows that conditions (H1)-(H3) are satisfied for Algorithm \ref{alg-IneSGM} with
\[
\epsilon>0, \quad \alpha_k\equiv 2L^{-\frac1p}, \quad \beta_k\equiv 1, \quad \gamma_k\equiv 1.
\]
Hence, the convergence theorems established in the preceding section can be directly applied to establish the convergence properties of the inexact subgradient method as follows.
\begin{theorem}\label{thm-sgmin}
Let $\{x_k\}$ be a sequence generated by Algorithm \ref{alg-IneSGM}.
\begin{enumerate}[{\rm (i)}]
  \item If $v_k\equiv v>0$, then $\liminf_{k\to \infty} f(x_k)\le f^*+L\left(\frac12 v\right)^p+\epsilon$.
  \item If $\{v_k\}$ is given by \eqref{eq-dim+stepsize}, then $\liminf_{k\to \infty} f(x_k)\le f^*+\epsilon$.
  \item If $\{v_k\}$ is given by
  \begin{equation}\label{eq-dyn-sgmin}
    v_k:= \lambda_k \left(\frac{f(x_k)-f^*-\epsilon}{L}\right)_+^{\frac1p},\quad \mbox{where $0<\underline{\lambda}\le \lambda_k \le \overline{\lambda}<2$},
  \end{equation}
  then $\lim_{k\to \infty} f(x_k)\le f^*+\epsilon$.
\end{enumerate}
\end{theorem}

\begin{theorem}\label{thm-cp-sgmin}
Let $\delta>0$, and let $\{x_k\}$ be a sequence generated by Algorithm \ref{alg-IneSGM}.
\begin{enumerate}[{\rm (i)}]
  \item Let $K_1:= \frac{L^{\frac1p}\dist^2(x_1,X^*)}{2 v \delta} $ and $v_k\equiv v>0$. Then
  $\min\limits_{1\le k\le K_1} f(x_k)\le f^*+\left(\frac12 L^{\frac1p}v+\delta\right)^p+\epsilon$.
  \item Let $K_2:=\left(\frac{(1-s)L^{\frac1p}\dist^2(x_1,X^*)}{2 c \delta}\right)^{\frac1{1-s}}$ and $\{v_k\}$ be given by \eqref{eq-dim+stepsize}. Then
  $\min\limits_{1\le k\le K_2} f(x_k)\le f^*+\left(\frac12 L^{\frac1p}ck^{-s}+\delta\right)^p+\epsilon$.
  \item Let $K_3:= \frac{L^{\frac2p}\dist^2(x_1,X^*)}{\underline{\lambda}(2-\overline{\lambda})\delta^2} $ and $\{v_k\}$ be given by \eqref{eq-dyn-sgmin}. Then $\min\limits_{1\le k\le K_3} f(x_k)\le f^*+\delta^p+\epsilon$.
\end{enumerate}
\end{theorem}

\begin{theorem}\label{thm-cr-sgmin}
Let $\{x_k\}$ be a sequence generated by Algorithm \ref{alg-IneSGM}. Suppose that $X^*$ is the set of boundedly weak sharp minima of order $q$ for $f$ over $X$ with modulus $\eta$.
\begin{enumerate}[{\rm (I)}]
  \item Suppose that $v_k\equiv v>0$ and that $f$ is coercive.
      \begin{enumerate}[{\rm (i)}]
      \item If $q= 2p$, then either $x_k\in X^*_\epsilon$ for some $k\in \IN$ or there exist $\tau\in [0,1)$ and $N\in \IN$ such that, for each $k\in \IN$,
        \[
        \dist^2(x_{k+N},X^*)\le \tau^k \dist^2(x_N,X^*) +2^{\frac1p-1}\eta^{-\frac1{p}}  \left(\epsilon^{\frac1p}+\frac12 L^{\frac1p}v\right).
        \]
      \item If $q>2p$ and $\epsilon^{\frac1p}+\frac12 L^{\frac1p}v<\eta^{-\frac{2}{q-2p}}(\frac{p}{vq}L^{\frac1p})^{\frac{q}{q-2p}}$, then either $x_k\in X^*_\epsilon$ for some $k\in \IN$ or there exist $\tau\in (0,1)$ and $N\in \IN$ such that, for each $k\in \IN$,
        \[
        \dist^2(x_{k+N},X^*)\le \tau^k \dist^2(x_N,X^*) +2^{\frac2q-\frac{2p}q}{\eta}^{-\frac{2}q} \left(\epsilon^{\frac1p}+\frac12 L^{\frac1p}v\right)^{\frac{2p}q}.
        \]
    \end{enumerate}
  \item Suppose that $\{v_k\}$ is given by \eqref{eq-dim+stepsize}, $q=2p$ and that $f$ is coercive. Then, either $x_k\in X^*_\epsilon$ for some $k\in \IN$ or there exist $C>0$ and $\tau\in (0,1)$ such that
  \begin{equation*}
    \dist^2(x_k,X^*)\le C\tau^k +\left(\frac{2\epsilon}{\eta}\right)^{\frac1p}\quad \mbox{for each } k\in \IN.
    \end{equation*}
  \item Suppose that $\{v_k\}$ is given by \eqref{eq-dyn-sgmin}.
    \begin{enumerate}
      \item[{\rm (i)}] If $q= p$, then either $x_k\in X^*_\epsilon$ for some $k\in \IN$ or there exist $\tau\in [0,1)$ and $N\in \IN$ such that
        \[
        \dist^2(x_{k+N},X^*)\le \tau^k \dist^2(x_N,X^*) +2^{\frac2p-1}\eta^{-\frac2{p}}\epsilon^{\frac2p} \quad \mbox{for each }  k\in \IN.
        \]
        \item[{\rm (ii)}] If $q>p$ and $\epsilon<\left(\frac{L}{\eta}\right)^{\frac{q}{q-p}}\left(\frac{q}p\underline{\lambda}(2-\overline{\lambda})\right)^{-\frac{pq}{2(q-p)}}$, then either $x_k\in X^*_\epsilon$ for some $k\in \IN$ or there exist $\tau\in (0,1)$ and $N\in \IN$ such that
    \[
    \dist^2(x_{k+N},X^*)\le \tau^k \dist^2(x_N,X^*) +2^{\frac2q-\frac{p}q}\eta^{-\frac2q} \epsilon^{\frac{2}q}\quad \mbox{for each } k\in \IN.
    \]
    \end{enumerate}
\end{enumerate}
\end{theorem}

\subsection{Conditional subgradient method}
The standard subgradient method (i.e., Algorithm \ref{alg-SGM}) usually suffers from a zig-zagging phenomenon and sustains a slow convergence in practical applications. To avoid the zig-zagging phenomenon and speed up the convergence behavior,  an idea of conditional subgradient method was proposed for either convex optimization \cite{LarssonPS96} or quasi-convex optimization problems \cite{HuJNCA16b}, which is stated as follows. It is worth mentioning that the algorithmic procedure of the conditional subgradient method is totally different from the conditional gradient method (also named the Franke-Wolfe method) \cite{CGM2018}, although they share a similar name.

\begin{algorithm}\label{alg-CondSGM}
{\rm
Select an initial point $x_1\in \R^n$ and a sequence of stepsizes $\{v_k\}\subseteq (0,+\infty)$. For each $k\in \IN$, given $x_k$, we calculate
\begin{equation*}\label{eq-CondSGM-epsilon-0}
g_k\in \partial^* f(x_k)\cap \mathbf{S} \quad \mbox{and} \quad
\mu_k\in
\left\{\begin{matrix}
        N_X(x_k)\cap \mathbf{S},&{\mbox{if }x_k\notin {\rm int} X,}\\
       \{0\}, &{\mbox{if }x_k\in {\rm int} X,}
    \end{matrix}\right.
\end{equation*}
and update $x_{k+1}$ by
\begin{equation*}\label{eq-CondSGM-epsilon}
x_{k+1}:=\mathbb{P}_X (x_k-v_k (g_k+\mu_k)).
\end{equation*}
}
\end{algorithm}

It was proved in \cite[Lemma 3.2]{HuJNCA16b} the condition subgradient method satisfies the basic inequality as follow under the blanket assumption:
\[
\| x_{k+1}-x^* \|^2 \le \| x_k-x^* \|^2 - 2v_k \left(\frac{f(x_k)-f^*}{L}\right)^\frac1p + 4{v_k}^2
\]
whenever $f(x_k)>f^*$, and $\| x_{k+1}-x_k\|\le 2v_k$. This shows that conditions (H1)-(H3) are satisfied with
\[
\epsilon=0, \quad \alpha_k\equiv 2L^{-\frac1p}, \quad \beta_k\equiv 4, \quad \gamma_k\equiv 2.
\]
Hence, the convergence theorems established in the preceding section can be directly applied to the conditional subgradient method as follows.
\begin{theorem}\label{thm-sgmcg}
Let $\{x_k\}$ be a sequence generated by Algorithm \ref{alg-CondSGM}.
\begin{enumerate}[{\rm (i)}]
  \item If $v_k\equiv v>0$, then $\liminf_{k\to \infty} f(x_k)\le f^*+L\left(2 v\right)^p$.
  \item If $\{v_k\}$ is given by \eqref{eq-dim+stepsize}, then $\liminf_{k\to \infty} f(x_k)\le f^*$.
  \item If $\{v_k\}$ is given by
  \begin{equation}\label{eq-dyn-sgmcg}
    v_k:= \frac{\lambda_k}4 \left(\frac{f(x_k)-f^*}{L}\right)^{\frac1p},\quad \mbox{where $0<\underline{\lambda}\le \lambda_k \le \overline{\lambda}<2$},
  \end{equation}
  then $\lim_{k\to \infty} f(x_k)= f^*$.
\end{enumerate}
\end{theorem}

\begin{theorem}\label{thm-cp-sgmcg}
Let $\delta>0$, and let $\{x_k\}$ be a sequence generated by Algorithm \ref{alg-CondSGM}.
\begin{enumerate}[{\rm (i)}]
  \item Let $K_1:= \frac{L^{\frac1p}\dist^2(x_1,X^*)}{2 v \delta} $ and $v_k\equiv v>0$. Then
  $\min\limits_{1\le k\le K_1} f(x_k)\le f^*+\left(2 L^{\frac1p}v+\delta\right)^p$.
  \item Let $K_2:=\left(\frac{(1-s)L^{\frac1p}\dist^2(x_1,X^*)}{2 c \delta}\right)^{\frac1{1-s}}$ and $\{v_k\}$ be given by \eqref{eq-dim+stepsize}. Then
  $\min\limits_{1\le k\le K_2} f(x_k)\le f^*+\left(2 L^{\frac1p}ck^{-s}+\delta\right)^p$.
  \item Let $K_3:= \frac{4L^{\frac2p}\dist^2(x_1,X^*)}{\underline{\lambda}(2-\overline{\lambda})\delta^2} $ and $\{v_k\}$ be given by \eqref{eq-dyn-sgmcg}. Then $\min\limits_{1\le k\le K_3} f(x_k)\le f^*+\delta^p$.
\end{enumerate}
\end{theorem}

\begin{theorem}\label{thm-cr-sgmcg}
Let $\{x_k\}$ be a sequence generated by Algorithm \ref{alg-CondSGM}. Suppose that $X^*$ is the set of boundedly weak sharp minima of order $q$ for $f$ over $X$ with modulus $\eta$.
\begin{enumerate}[{\rm (I)}]
  \item Suppose that $v_k\equiv v>0$ and that $f$ is coercive.
      \begin{enumerate}[{\rm (i)}]
      \item If $q= 2p$, then either $x_k\in X^*$ for some $k\in \IN$ or there exist $\tau\in [0,1)$ and $N\in \IN$ such that
        \[
        \dist^2(x_{k+N},X^*)\le \tau^k \dist^2(x_N,X^*) +2^{\frac1p}\left(\frac{L}{\eta}\right)^{\frac1p}v \quad \mbox{for each } k\in \IN.
        \]
      \item If $q>2p$ and $v<\left(2^{p}\frac{L}{\eta}(\frac{p}{2q})^{\frac{q}{2}}\right)^{\frac{1}{q-p}}$, then either $x_k\in X^*$ for some $k\in \IN$ or there exist $\tau\in (0,1)$, $c>0$ and $N\in \IN$ such that
        \[
        \dist^2(x_{k+N},X^*)\le \tau^k \dist^2(x_N,X^*) +2^{\frac2q}\left(\frac{L}{\eta}\right)^{\frac2q}v^{\frac{2p}q}\quad \mbox{for each } k\in \IN.
        \]
    \end{enumerate}
    \item Suppose that $\{v_k\}$ is given by \eqref{eq-dim+stepsize}, $q=2p$ and that $f$ is coercive. Then, either $x_k\in X^*$ for some $k\in \IN$ or there exists $N\in \IN$ such that
    \begin{equation*}\label{eq-cr-con2}
    \dist^2(x_k,X^*)\le 2c\left(\frac{2L}\eta\right)^{\frac1p}k^{-s}\quad \mbox{for each } k\ge N.
    \end{equation*}
  \item Suppose that $\{v_k\}$ is given by \eqref{eq-dyn-sgmcg}.
    \begin{enumerate}
      \item[{\rm (i)}] If $q= p$, then there exist $\tau\in [0,1)$ and $N\in \IN$ such that
        \[
        \dist^2(x_{k+N},X^*)\le \tau^k \dist^2(x_N,X^*) \quad \mbox{for each }  k\in \IN.
        \]
        \item[{\rm (ii)}] If $q>p$, then there exist $\gamma>0$ and $N\in \IN$ such that
          \[
          \dist^2(x_{k+N},X^*)\le \frac{\dist^2(x_N,X^*)}{(1+\gamma k)^{\frac{p}{q-p}}}\quad \mbox{for each } k\in \IN.
          \]
    \end{enumerate}
\end{enumerate}
\end{theorem}

\section*{Acknowledgment}
The authors are grateful to  the editor and anonymous reviewers for their valuable suggestions and remarks which allow us to improve the original presentation of the paper.

%

\end{document}